\documentclass[english]{article}
\usepackage[T1]{fontenc}
\usepackage[latin9]{inputenc}
\usepackage{amsthm}
\usepackage{amsmath}
\usepackage{amssymb}
\usepackage{esint}

\makeatletter

\theoremstyle{plain}
\newtheorem{thm}{Theorem}
  \theoremstyle{remark}
  \newtheorem{rem}[thm]{Remark}
  \theoremstyle{plain}
  \newtheorem{prop}[thm]{Proposition}
  \theoremstyle{plain}
  \newtheorem{lem}[thm]{Lemma}
\newcommand{\lyxaddress}[1]{
\par {\raggedright #1
\vspace{1.4em}
\noindent\par}
}

\usepackage{commath}

\makeatother

\usepackage{babel}

\begin{document}

\title{Asymptotic Analysis of a Non-Linear Non-Local Integro-Differential
Equation Arising from\\Bosonic Quantum Field Dynamics}

\author{S\'ebastien Breteaux%
\thanks{Institut f\"ur Analysis und Algebra, Technische Universit\"at Braunschweig,
Rebenring 31, A 14, 38106 Braunschweig, Germany. \emph{E-mail adress}:
sebastien.breteaux@ens-cachan.org.%
}{ }%
\thanks{This work was partially done in the IRMAR laboratory of the University
of Rennes 1, Campus de Beaulieu, bat. 22 \& 23, 263 avenue du G\'en\'eral
Leclerc, CS 74205, 35042 Rennes C\'edex, France.%
}}
\maketitle
\begin{abstract}
We introduce a one parameter family of non-linear, non-local integro-differential
equations and its limit equation. These equations originate from a
derivation of the linear Boltzmann equation using the framework of
bosonic quantum field theory. We show the existence and uniqueness
of strong global solutions for these equations, and a result of uniform
convergence on every compact interval of the solutions of the one
parameter family towards the solution of the limit equation.
\end{abstract}
\phantom{qsdf}

\emph{Keywords}: Integro-differential equation, nonlinear equation,
nonlocal equation, equation with memory.

\emph{Mathematics Subject Classification (2010)}: 34G20, 45J05, 47G20.

\section{Introduction}

Let $d\geq1$ be an integer, $h$ a strictly positive parameter, $f$
a function in the Schwartz space $\mathcal{S}(\mathbb{R}^{d};\mathbb{C})$
and $\vec{\xi}_{0}$ a vector in $\mathbb{R}^{d}\setminus\{0\}$.

The spaces of bounded and continuous functions from $X\subset\mathbb{R}$
to $Y\subset\mathbb{R}^{d}$ are denoted by $\mathcal{B}(X;Y)$ and
$\mathcal{C}(X;Y)$. For a function $\vec{u}\in\mathcal{B}(X;Y)$
we write $\|\vec{u}\|_{\infty,X^{\prime}}=\sup\{|\vec{u}(x)|\,,\, x\in X^{\prime}\}$
for $X^{\prime}\subset X$.

The applications $\mathcal{F}^{(h)}$ and $F^{(0)}$ from $\mathcal{B}(\mathbb{R}^{+};\mathbb{R}^{d})$
to $\mathcal{C}(\mathbb{R}^{+};\mathbb{R}^{d})$ are defined for $\vec{u}\in\mathcal{B}(\mathbb{R}^{+};\mathbb{R}^{d})$
by\begin{align*}
\mathcal{F}^{(h)}(\vec{u})(t) & :=-2\Re\int_{\mathbb{R}^{d}}\int_{0}^{t/h}e^{-ir\left(\vec{\eta}^{.2}-2\vec{\eta}.\fint_{t-hr}^{t}\vec{u}_{\sigma}\dif\sigma\right)}\vec{\eta}|f(\vec{\eta})|^{2}\dif r\dif\vec{\eta}\\
F^{(0)}(\vec{u})(t) & :=-2\Re\int_{0}^{\infty}\int_{\mathbb{R}^{d}}e^{-ir\left(\vec{\eta}^{.2}-2\vec{\eta}.\vec{u}_{t}\right)}\vec{\eta}|f(\vec{\eta})|^{2}\dif\vec{\eta}\dif r\end{align*}
where $\fint_{a}^{b}\vec{u}(\sigma)\dif\sigma:=\frac{1}{b-a}\int_{a}^{b}\vec{u}(\sigma)\dif\sigma$.
\begin{thm}
For every $h>0$, the non-linear integro-differential equation\begin{align}
\left\{ \begin{aligned}\frac{\dif}{\dif t}\vec{\xi}_{t}^{\,(h)} & =\mathcal{F}^{(h)}(\vec{\xi}_{t}^{\,(h)})\\
\vec{\xi}_{t=0}^{\,(h)} & =\vec{\xi}_{0}\,,\end{aligned}
\right.\label{eq:equation-de-depart}\end{align}
admits a unique solution $\vec{\xi}^{\,(h)}$ in $\mathcal{C}^{1}(\mathbb{R}_{t}^{+};\mathbb{R}^{d})$.

The limit equation as $h\to0$\begin{align}
\left\{ \begin{aligned}\frac{\dif}{\dif t}\vec{\xi}_{t}^{\,(0)} & =F^{(0)}(\vec{\xi}_{t}^{\,(0)})\\
\vec{\xi}_{t=0}^{\,(0)} & =\vec{\xi}_{0}\,.\end{aligned}
\right.\label{eq:equation-limite1}\end{align}
admits a unique maximal solution $\vec{\xi}^{\,(0)}$ in $\mathcal{C}^{1}([0,T_{max});\mathbb{R}^{d})$
with $T_{max}>0$.
\begin{itemize}
\item The norm of $\vec{\xi}^{\,(0)}$ decreases with time.
\item If $\min\{|f(\vec{\eta})|,\,\vec{\eta}\in\bar{B}(0,2|\vec{\xi}_{0}^{\,(0)}|)\}$
is strictly positive then $T_{max}=+\infty$, and $\vec{\xi}_{t}^{\,(0)}\to0$
as $t\to+\infty$.
\end{itemize}
For any $T\in(0,T_{max})$, $\vec{\xi}^{\,(h)}$ converges uniformly
to $\vec{\xi}^{\,(0)}$ on $[0,T]$ as $h\to0$.\end{thm}
\begin{rem}
The non-locality feature of Equation~(\ref{eq:equation-de-depart})
disapears in the limit Equation~(\ref{eq:equation-limite1}).
\end{rem}

\paragraph{Origin of the equations}

These equations come from a problem of derivation of the linear Boltzmann
equation in dimension~$d$ from a model with a particle in a Gaussian
random field (centered and invariant by translation), in the weak
density limit, as presented in~\cite{arXiv:1107.0788}. The parameter
$h$ has then the interpretation of the ratio of the microscopic typical
length over the macroscopic one, and it is thus natural to be interested
in the behaviour of the equations as $h\to0$. Note that the same
scaling is used to rescal the time. It is \emph{not} the Planck constant
(here taken equal to one by a suitable choice of units). It turns
out that this stochastic problem can be translated in a deterministic
one with more geometric features by using an isomorphism between the
$L^{2}(\Omega)$ space associated with the Gaussian random field and
the symmetric Fock space $\Gamma L^{2}=\bigoplus_{n=0}^{\infty}(L^{2})^{\otimes_{s}n}$
over $L^{2}:=L^{2}(\mathbb{R}^{d};\mathbb{C})$ (see~\cite{MR0489552}
for information about this isomorphism). Up to some isomorphisms the
Schrödinger equation can be rewritten as\[
ih\,\partial_{t}u=Q(z)^{Wick}\, u,\]
where
\begin{itemize}
\item The function $u$ has values in~$\Gamma L^{2}$.
\item The polynomial~$Q$ in the variable $z\in L^{2}$ is defined by\[
Q(z)=\xi^{2}+\langle z,(\vec{\eta}^{.2}-2\vec{\xi}.\vec{\eta})z\rangle+\langle z,\vec{\eta}z\rangle^{.2}+2\sqrt{h}\Re\langle z,f\rangle\,,\]
for a fixed $\vec{\xi}$ in $\mathbb{R}^{d}\setminus\{0\}$, with
$\vec{a}^{.2}=\sum_{j=1}^{d}a_{j}^{2}$ for every vector~$\vec{a}$,
$\langle\cdot,\cdot\rangle$ is the scalar product in~$L^{2}$, $\Re$
denotes the real part. $\vec{\eta}$ denotes the mutliplication operator
which to a function $v$ of the variable $\vec{\eta}\in\mathbb{R}^{d}$
associates the function with $d$ components $\vec{\eta}\mapsto\vec{\eta}v(\eta)$.
We do here as if this operator was bounded on $L^{2}$ since the small
improvements needed to handle this case are irrelevant for this article.
\item The Wick quantization is defined for a monomial $b(z)=\langle z^{\otimes q},\tilde{b}z^{\otimes p}\rangle$
with variable in $z\in L^{2}$ and $\tilde{b}\in\mathcal{L}((L^{2})^{\otimes_{s}p},(L^{2})^{\otimes_{s}q})$
by the fomula\[
b^{Wick}\big|_{L^{2}(\mathbb{R}^{d})^{\otimes_{s}p+n}}=\frac{\sqrt{(n+p)!(n+q)!}}{n!}\;\tilde{b}\otimes_{s}Id_{L^{2}(\mathbb{R}^{d})^{\otimes_{s}n}}\,.\]

\end{itemize}
In this framework the Hamilton equations associated with the polynomial
$Q(z)$ (see~\cite{MR2465733,MR2513969,MR2802894}) are $ih\partial_{t}z_{t}^{(h)}=\partial_{\bar{z}}Q(z_{t}^{(h)})$
with $\partial_{\bar{z}}Q(z)=(\vec{\eta}^{.2}-2\vec{\xi}.\vec{\eta})z+2\langle z,\vec{\eta}z\rangle.\vec{\eta}z+\sqrt{h}f$
which can again be written as\[
ih\partial_{t}z_{t}^{(h)}=(\vec{\eta}^{.2}-2\vec{\xi}_{t}.\vec{\eta})z_{t}^{(h)}+\sqrt{h}f\]
once we defined $\vec{\xi}_{t}^{\,(h)}=\vec{\xi}-\langle z_{t}^{(h)},\vec{\eta}z_{t}^{(h)}\rangle$.
We are interested in the solution of this equation with initial data
$z_{t=0}^{h}=0$ which is\[
z_{t}^{h}=-\frac{i}{h}\int_{0}^{t}e^{-\frac{i}{h}\int_{s}^{t}(\vec{\eta}^{.2}-2\vec{\xi}_{t}^{\,(h)}.\vec{\eta})d\sigma}\sqrt{h}f\, ds\,.\]
This solution is fully determined by $\vec{\xi}_{t}^{\,(h)}$ which
satisfies Equation~\ref{eq:equation-de-depart}. The study of equation
of $\vec{\xi}_{t}^{\,(h)}$ is thus justified by the study of the
hamiltonian equation associated with the polynom $Q(z)$ which is
itsel useful to give an approximate solution to the Schrödinger equation
above in terms of coherent states (see also the Hepp method, introduced
in~\cite{MR0332046} and presented in~\cite{MR2465733} whose viewpoint
we use in this article).

\paragraph{Intrisic significance of the equations}

Eventhough this approach did not until now allow us to improve our
results on the derivation of the linear Boltzmann equation, these
equations are interesting in themselve since it is a case of non-linear,
non-local equations and thus an already non trival problem for which
a thorough study is possible.

\paragraph{Organisation of the paper}

We show in Section~\ref{sec:parameter-dependent-equation} the existence
and uniqueness of a global solution to the parameter dependent Equation~(\ref{eq:equation-de-depart}),
along with some estimates about the solutions. Section~\ref{sec:limit-equation}
is devoted to the Equation~(\ref{eq:equation-limite1}). We first
compute a more geometric form of the application $F^{(0)}$ and then
use it to prove the existence and uniqueness of the solution to Equation~(\ref{eq:equation-limite1})
through the usual theory of differential equations. Then we compare
both solutions in Section~\ref{sec:Comparison} using a Grönwall
type argument.

\paragraph*{Notation}
\begin{enumerate}
\item For any time $T>0$ the applications $\mathcal{F}^{(h)}$ and $F^{(0)}$
induce applications from $\mathcal{B}([0,T];\mathbb{R}^{d})$ to $\mathcal{C}([0,T];\mathbb{R}^{d})$.
\item We sometime write $F^{(0)}(\vec{u})$ with $\vec{u}$ a fixed vector
in $\mathbb{R}^{d}\setminus\{0\}$ for \[
-2\Re\int_{0}^{\infty}\int_{\mathbb{R}^{d}}e^{-ir(\vec{\eta}^{.2}-2\vec{\eta}.\vec{u})}\vec{\eta}|f(\vec{\eta})|^{2}\dif\vec{\eta}\dif r\,.\]

\item We introduce the function $g:\vec{\eta}\in\mathbb{R}^{d}\mapsto\vec{\eta}|f(\vec{\eta})|^{2}$
to simplify the notations.
\end{enumerate}

\section{\label{sec:parameter-dependent-equation}Parameter-Dependent Equation}
\begin{prop}
In dimension $d\geq1$, for any $h>0$ Equation~(\ref{eq:equation-de-depart})
has a unique solution in $\mathcal{C}^{1}(\mathbb{R}^{+};\mathbb{R}^{d})$.\end{prop}
\begin{proof}
Suppose we know that Equation~(\ref{eq:equation-de-depart}) has
a unique solution $\vec{\xi}^{\,(h)}$ on an interval $[0,T]$ and
want to extend this solution to a larger interval $[0,T+\delta]$.
We thus consider the application~$\Phi$ from $\mathcal{C}^{1}([T,T+\delta];\mathbb{R}^{d})$
to itself defined by\[
\Phi(\vec{u})(t)=\vec{\xi}^{\,(h)}(T)+\int_{T}^{t}\mathcal{F}^{(h)}(\vec{\tilde{u}})(s)\dif s\,,\]
where the function $\vec{\tilde{u}}$ is defined on $[0,T+\delta]$
and coincide with $\vec{\xi}^{(h)}$ on $[0,T]$ and $\vec{\tilde{u}}=\vec{u}$
on $[T,T+\delta]$. We want to prove that for $\delta$ small enough
$\Phi$ is a contraction.

We then obtain\begin{align*}
|\Phi(\vec{u})^{\prime}(t)-\Phi(\vec{v})^{\prime}(t)| & =|\mathcal{F}^{(h)}(\vec{u})(t)-\mathcal{F}^{(h)}(\vec{v})(t)|\\
 & \leq2\negmedspace\int_{\mathbb{R}^{d}}\negmedspace\int_{0}^{t/h}\negmedspace\left|e^{ir2\vec{\eta}\fint_{t-hr}^{t}\vec{\tilde{u}}_{\sigma}\dif\sigma}-e^{ir2\eta\fint_{t-hr}^{t}\vec{\tilde{v}}_{\sigma}\dif\sigma}\right|\negmedspace|g(\vec{\eta})|\!\dif r\!\dif\vec{\eta}\\
 & \leq4\int_{\mathbb{R}^{d}}\int_{0}^{t/h}\big|r\vec{\eta}.\fint_{t-hr}^{t}(\vec{\tilde{u}}_{\sigma}-\vec{\tilde{v}}_{\sigma})\dif\sigma\big|\dif r|g(\vec{\eta})|\dif\vec{\eta}\\
 & \leq4C_{g}\int_{0}^{t/h}r\frac{t-T}{hr}\dif r\|\vec{u}-\vec{v}\|_{\infty}\\
 & \leq4\frac{t}{h^{2}}(t-T)C_{g}\|\vec{u}-\vec{v}\|_{\infty}\\
 & \leq4\frac{T+\delta}{h^{2}}(t-T)C_{g}\|\vec{u}-\vec{v}\|_{\infty}\end{align*}
with $C_{g}=\int_{\mathbb{R}^{d}}|\vec{\eta}||g(\vec{\eta})|\dif\vec{\eta}$.
Integrating $\Phi(\vec{u})^{\prime}(t)-\Phi(\vec{u})^{\prime}(t)$
in time and taking the modulus we also get the estimate\[
\|\Phi(\vec{u})-\Phi(\vec{v})\|_{\infty}\leq2\frac{T+\delta}{h^{2}}\delta^{2}C_{g}\|\vec{u}-\vec{v}\|_{\infty}\,.\]
We can control the norm of $\Phi(\vec{u})-\Phi(\vec{v})$ by using
$\|\vec{u}-\vec{v}\|_{\infty}\leq\|\vec{u}-\vec{v}\|_{\mathcal{C}^{1}}$
\[
\|\Phi(\vec{u})-\Phi(\vec{v})\|_{\mathcal{C}^{1}}\leq(4\delta+2\delta^{2})\frac{T+\delta}{h^{2}}C_{g}\|\vec{u}-\vec{v}\|_{\mathcal{C}^{1}}\,,\]
for $\delta(T+\delta)\leq\frac{h^{2}}{8C_{g}}$ we get a contraction,
and we can apply a classical fixed point theorem.

Thus to prove a global existence result it is enough to observe that,
for any positive constant $C$, the sequence $(t_{n})_{n}$ such that
$t_{0}=0$ and $(t_{n+1}-t_{n})^{2}+t_{n}(t_{n+1}-t_{n})=C$ with
positive increments is such that \[
t_{n+1}-t_{n}=\frac{\sqrt{t_{n}^{2}+4C}-t_{n}}{2}=\frac{2C}{\sqrt{t_{n}^{2}+4C}+t_{n}}\geq\frac{C}{t_{n}+\sqrt{C}}\,.\]
using $\sqrt{a+b}\leq\sqrt{a}+\sqrt{b}$ for $a$ and $b$ non-negative.
A sum over $n$ gives $t_{N+1}\geq\sum_{n=0}^{N}\frac{C}{t_{n}+\sqrt{C}}$
and this shows that $(t_{n})$ is necessarily unbounded. We thus obtain
the existence and the uniqueness of a solution on $\mathbb{R}^{+}$.\end{proof}
\begin{prop}
In dimension $d\geq3$, the family $(\vec{\xi}^{\,(h)\prime})_{h>0}$
is uniformly bounded, \emph{i.e.} \[
\|\vec{\xi}^{\,(h)\prime}\|_{\infty,[0,+\infty)}\leq2C_{g}^{1}\,,\]
with $C_{g}^{1}=\|g\|_{L^{1}}+(\frac{d}{2}-1)\pi^{d/2}\|\hat{g}\|_{L^{1}}$.
(The Fourier transform is defined by $\hat{u}(x)=\frac{1}{\sqrt{2\pi}^{d}}\int_{\mathbb{R}^{d}}e^{-ix.\eta}u(\vec{\eta})\dif\vec{\eta}$.)\end{prop}
\begin{proof}
It is enough to show that $\mathcal{F}^{\left(h\right)}$ is bounded.
Indeed, for $d\geq3$, using Parseval equality we get\begin{align*}
|\mathcal{F}^{(h)}(u)(t)| & \leq2\int_{0}^{t/h}\Big|\int_{\mathbb{R}^{d}}e^{-ir(\vec{\eta}^{.2}-2\vec{\eta}\fint_{t-hr}^{t}\vec{u}_{\sigma}\dif\sigma)}g(\vec{\eta})\dif\vec{\eta}\Big|\dif r\\
 & \leq2\left\Vert g\right\Vert _{L^{1}}+2\int_{1}^{t/h}\Big|\int_{\mathbb{R}^{d}}e^{-ir\vec{\eta}^{.2}}g\Big(\vec{\eta}+\fint_{t-hr}^{t}\vec{u}_{\sigma}\dif\sigma\Big)\dif\vec{\eta}\Big|\dif r\\
 & \leq2\left\Vert g\right\Vert _{L^{1}}+2\int_{1}^{t/h}\Big(\frac{\pi}{r}\Big)^{d/2}\left\Vert \hat{g}\right\Vert _{L^{1}}\dif r\\
 & \leq2\left\Vert g\right\Vert _{L^{1}}+(d-2)\pi^{d/2}\left\Vert \hat{g}\right\Vert _{L^{1}}\,.\tag*{\qedhere}\end{align*}

\end{proof}

\section{\label{sec:limit-equation}Limit Equation}
\begin{prop}
\label{pro:formula-for-F(0)}For any $\vec{u}$ in $\mathbb{R}^{d}\setminus\left\{ 0\right\} $,
with $\hat{\vec{u}}=\frac{\vec{u}}{\left|\vec{u}\right|}$, \begin{align*}
F^{(0)}(\vec{u}) & =-\pi|\vec{u}|^{d-1}\int_{[0,2]}\int_{S^{d-2}}(\rho\hat{\vec{u}}+\sqrt{2\rho-\rho^{2}}\vec{\omega}')\\
 & \phantom{=}\times\big|f\big(|\vec{u}|(\rho\hat{u}+\sqrt{2\rho-\rho^{2}}\vec{\omega}')\big)\big|^{2}\,\dif\mathcal{H}^{d-2}(\vec{\omega}')\,\sqrt{2\rho-\rho^{2}}^{d-2}\dif\rho\,,\end{align*}
where $\vec{\omega}'$ is in the orthogonal of $\vec{u}$.\end{prop}
\begin{proof}
We introduce a parameter $\varepsilon$ to permute the integrals\begin{align*}
F^{(0)}\left(\vec{u}\right) & =\lim_{\varepsilon\to0^{+}}-2\Re\int_{0}^{+\infty}e^{-r\varepsilon}\int_{\mathbb{R}_{\eta}^{d}}e^{-ir(\vec{\eta}^{.2}-2\vec{\eta}.\vec{u})}g(\vec{\eta})\dif\vec{\eta}\dif r\\
 & =\lim_{\varepsilon\to0^{+}}-2\Re\int_{0}^{+\infty}\int_{\mathbb{R}_{\eta}^{d}}e^{-ir(\vec{\eta}^{.2}-2\vec{\eta}.\vec{u})-r\varepsilon}g(\vec{\eta})\dif\vec{\eta}\dif r\\
 & =\lim_{\varepsilon\to0^{+}}-2\Re\int_{\mathbb{R}_{\eta}^{d}}\int_{0}^{+\infty}e^{-ir(\vec{\eta}^{.2}-2\vec{\eta}.\vec{u})-r\varepsilon}\dif rg(\vec{\eta})\dif\vec{\eta}\\
 & =\lim_{\varepsilon\to0^{+}}-2\int_{\mathbb{R}_{\eta}^{d}}\Big(\Re\int_{0}^{+\infty}e^{-r\left[i(\vec{\eta}^{.2}-2\vec{\eta}.\vec{u})+\varepsilon\right]}\dif r\Big)g(\vec{\eta})\dif\vec{\eta}\,.\end{align*}
The time integral can be explicitly computed\[
\int_{0}^{+\infty}e^{-r\left[i(\vec{\eta}^{.2}-2\vec{\eta}.\vec{u})+\varepsilon\right]}\dif r=\frac{\varepsilon-i(\vec{\eta}^{.2}-2\vec{\eta}.\vec{u})}{(\vec{\eta}^{.2}-2\vec{\eta}.\vec{u})^{2}+\varepsilon^{2}}\]
and taking the real part and the limit as $\varepsilon\to0^{+}$ we
obtain \[
\lim_{\varepsilon\to0^{+}}\Re\int_{0}^{+\infty}e^{-r\left[i(\vec{\eta}^{.2}-2\vec{\eta}.\vec{u})+\varepsilon\right]}\dif r=\pi\delta(\vec{\eta}^{.2}-2\vec{\eta}.\vec{u})\,.\]
Thus\begin{align*}
F^{(0)}(\vec{u}) & =-2\pi\int_{\mathbb{R}_{\eta}^{d}}\delta(\vec{\eta}^{.2}-2\vec{\eta}.\vec{u})g(\vec{\eta})\dif\vec{\eta}\\
 & =-2\pi\int_{\vec{u}+|\vec{u}|S^{d-1}}g(\vec{\eta})\frac{\dif\mathcal{H}^{d-1}(\vec{\eta})}{\left|\nabla(\vec{\eta}^{.2}-2\vec{\eta}.\vec{u})\right|}\\
 & =-\pi|\vec{u}|^{-1}\int_{\vec{u}+|\vec{u}|S^{d-1}}g(\vec{\eta})\dif\mathcal{H}^{d-1}(\vec{\eta})\end{align*}
with $\dif\mathcal{H}^{d-1}$ the $(d-1)$-dimensional Hausdorff measure.
Then, using $\hat{\vec{u}}=\vec{u}/|\vec{u}|$,\begin{align}
F^{(0)}(\vec{u}) & =-\pi|\vec{u}|^{-1}\int_{\vec{u}+|\vec{u}|S^{d-1}}\vec{\eta}|f(\vec{\eta})|^{2}\dif\mathcal{H}^{d-1}(\vec{\eta})\nonumber \\
 & =-\pi|\vec{u}|^{-1}\int_{\hat{\vec{u}}+S^{d-1}}|\vec{u}|\vec{\eta}'\,|f(|\vec{u}|\vec{\eta}')|^{2}|\vec{u}|^{d-1}\dif\mathcal{H}^{d-1}(\vec{\eta}')\nonumber \\
 & =-\pi|\vec{u}|^{d-1}\int_{\hat{\vec{u}}+S^{d-1}}\vec{\eta}|f(|\vec{u}|\vec{\eta})|^{2}\dif\mathcal{H}^{d-1}(\vec{\eta})\label{eq:F(0)-troisieme-formule}\end{align}
we can now parametrize the sphera in a fashion which shows the particular
importance of the direction of~$\vec{u}$. We define the change of
variable\begin{align*}
[0,2]\times S^{d-2} & \to\hat{\vec{u}}+S^{d-1}\\
(\rho,\vec{\omega}') & \mapsto\rho\hat{\vec{u}}+\sqrt{2\rho-\rho^{2}}\vec{\omega}'\end{align*}
whose inverse transformation is given by\[
\rho=\vec{\eta}.\hat{\vec{u}}\qquad\mbox{and}\qquad\vec{\omega}'=\frac{\vec{\eta}-(\vec{\eta}.\hat{\vec{u}})\hat{\vec{u}}}{|\vec{\eta}-(\vec{\eta}.\hat{\vec{u}})\hat{\vec{u}}|}\,.\]
The jacobian under this change of variable is $\sqrt{2\rho-\rho^{2}}^{d-2}$
as $|\vec{\eta}-(\vec{\eta}.\hat{\vec{u}})\hat{\vec{u}}|=\sqrt{2\rho-\rho^{2}}$
and thus we get the result.\end{proof}
\begin{prop}
In dimension $d\geq3$ the limit differential equation~(\ref{eq:equation-limite1})
has a unique maximal solution~$\vec{\xi}^{\,(0)}$.

The norm of this solution decreases with the time.

If $\min\{|f(\vec{\eta})|,\,\vec{\eta}\in\bar{B}(0,2|\vec{\xi}_{0}^{\,(0)}|)\}$
is strictly positive then this solution is defined on $\mathbb{R}^{+}$,
and $\vec{\xi}_{t}^{\,(0)}$ converges to $0$ as $t\to+\infty$.\end{prop}
\begin{proof}
We first show that on any compact subset of $\mathbb{R}^{d}\setminus\{0\}$
the map $F^{(0)}$ is Lipschitz so that the theorem of Cauchy-Lipschitz
applies. Indeed on $C(r,R)\times\bar{B}(0,2)$, with $C(r,R)=\{\vec{x}\in\mathbb{R}^{d},\, r\leq|\vec{x}|\leq R\}$
and $0<r<R$, the application\[
(\vec{u},\vec{\eta})\mapsto\vec{\eta}|f(|\vec{u}|\vec{\eta})|^{2}|\vec{u}|^{d-1}\]
is smooth and thus Lipschitz with a constant $L$, and\begin{align*}
|F^{(0)}(\vec{u})-F^{(0)}(\vec{v})| & =\pi\left||\vec{u}|^{d-1}\int_{\hat{\vec{u}}+S^{d-1}}\vec{\eta}|f(|\vec{u}|\vec{\eta})|^{2}\dif\mathcal{H}^{d-1}(\vec{\eta})\right.\\
 & \phantom{=\pi}\qquad\left.-|\vec{v}|^{d-1}\int_{\hat{\vec{v}}+S^{d-1}}\vec{\eta}|f(|\vec{v}|\vec{\eta})|^{2}\dif\mathcal{H}^{d-1}(\vec{\eta})\right|\\
 & \leq\pi L(|\hat{\vec{u}}-\hat{\vec{v}}|+|\vec{u}-\vec{v}|)\\
 & \leq L'|\vec{u}-\vec{v}|\end{align*}
where we used an equivalent formulas for $F^{(0)}$ derived in Equation~(\ref{eq:F(0)-troisieme-formule})
in the proof of Proposition~\ref{pro:formula-for-F(0)}. Thus there
exists a unique solution $\vec{\xi}^{\,(0)}$ defined on a maximal
interval $[0,T_{\max})$.

We then prove that the norm of this solution is decreasing, indeed
\begin{align*}
\frac{\dif}{\dif t}\frac{1}{2}\big|\vec{\xi}_{t}^{\,(0)}\big|^{2} & =\vec{\xi}_{t}^{\,(0)}.\vec{\xi}_{t}^{\,(0)\prime}\\
 & =\vec{\xi}_{t}^{\,(0)}.F^{(0)}\big(\vec{\xi}_{t}^{\,(0)}\big)\\
 & =-\pi\big|\vec{\xi}_{t}^{\,(0)}\big|^{d-1}\int_{[0,2]}\int_{S^{d-2}}\rho\big(\hat{\vec{\xi}}_{t}^{\,(0)}+\sqrt{2\rho-\rho^{2}}\vec{\omega}'\big).\vec{\xi}_{t}^{\,(0)}\\
 & \phantom{=}\quad\big|f\big(\big|\vec{\xi}_{t}^{\,(0)}\big|(\rho\vec{u}'+\sqrt{2\rho-\rho^{2}}\vec{\omega}')\big)\big|^{2}\dif\mathcal{H}^{d-2}(\vec{\omega}')\,\sqrt{2\rho-\rho^{2}}^{d-2}\dif\rho\\
 & =-\pi\big|\vec{\xi}_{t}^{\,(0)}\big|^{d}\int_{\left[0,2\right]}\int_{S^{d-2}}\rho\big|f\big(\big|\vec{\xi}_{t}^{\,(0)}\big|(\rho\vec{u}'+\sqrt{2\rho-\rho^{2}}\vec{\omega}')\big)\big|^{2}\\
 & \phantom{=}\qquad\qquad\qquad\qquad\qquad\qquad\qquad\dif\mathcal{H}^{d-2}(\vec{\omega}')\,\sqrt{2\rho-\rho^{2}}^{d-2}\dif\rho\\
 & \leq0\,.\end{align*}
 Thus $\vec{\xi}^{\,(0)}$ is necessarily bounded and cannot blow
up. The maximal interval is thus necessarily $[0,+\infty)$.

If $f$ does not vanish we have even more information, we know that\[
-C_{M,f}\big|\vec{\xi}_{t}^{\,(0)}\big|^{d}\leq\od{}t\big|\vec{\xi}_{t}^{\,(0)}\big|^{2}\leq-C_{m,f}\big|\vec{\xi}_{t}^{\,(0)}\big|^{d}\]
where $C_{M,f}=C_{d}\max_{\bar{B}(0,2|\vec{\xi}^{\,(0)}|)}|f|$ and
$C_{m,f}=C_{d}\min_{\bar{B}(0,2|\vec{\xi}^{\,(0)}|)}|f|$ with $C_{d}=2\pi\int_{[0,2]}\rho\sqrt{2\rho-\rho^{2}}^{d-2}\dif\rho\mathcal{H}^{d-2}(S^{d-2})$.
By an integration we conclude that $\big|\vec{\xi}_{t}^{\,(0)}\big|^{2}$
is in the interval bounded by the quantities \[
\big(|\vec{\xi}_{0}|^{1-d/2}+(\frac{d}{2}-1)Ct\big)^{-\frac{2}{d-2}}\]
for $C=C_{M,f}$ and $C_{m,f}$. Thus $\vec{\xi}_{t}^{\,(0)}$ doesn't
exit from any compact set of $\mathbb{R}^{d}\setminus\{0\}$ in finite
time, and this \emph{a priori} estimate implies $T_{\max}=+\infty$.
\end{proof}
\begin{rem}
\label{rem:Moy_Xi-Xi}We have the control, for $0<hr<t$, \[
\Big|\fint_{t-hr}^{t}\vec{\xi}^{\,(h)}(\sigma)\dif\sigma-\vec{\xi}^{\,(h)}(t)\Big|\leq\frac{1}{2}\big\|\vec{\xi}^{\,(h)\prime}\big\|_{\infty,[t-hr,t]}hr\leq C_{g}^{1}hr\,.\]
\end{rem}
\begin{proof}
Indeed\[
\fint_{t-hr}^{t}\vec{\xi}^{\,(h)}\left(\sigma\right)\dif\sigma-\vec{\xi}^{\,(h)}\left(t\right)=\fint_{t-hr}^{t}\int_{t}^{\sigma}\vec{\xi}^{\,(h)\prime}\left(v\right)\dif v\dif\sigma\,,\]
and by taking the norm we get\[
\Big\|\fint_{t-hr}^{t}\vec{\xi}^{\,(h)}(\sigma)\dif\sigma-\vec{\xi}^{\,(h)}(t)\Big\|\leq\big\|\vec{\xi}^{\,(h)\prime}\big\|_{\infty,[t-hr,t]}\fint_{t-hr}^{t}(t-\sigma)\dif\sigma\]
which gives the result.
\end{proof}

\section{\label{sec:Comparison}Control of the Difference Between the Two
Equations}

We set the application $F^{(h)}$ from $\mathcal{B}(\mathbb{R}^{+};\mathbb{R}^{d})$
to $\mathcal{C}(\mathbb{R}^{+};\mathbb{R}^{d})$, defined for $\vec{u}\in\mathcal{B}(\mathbb{R}^{+};\mathbb{R}^{d})$
by\[
F^{(h)}(\vec{u})(t)=-2\Re\int_{\mathbb{R}^{d}}\int_{0}^{t/h}e^{-ir\left(\vec{\eta}^{.2}-2\vec{\eta}.\vec{u}_{t}\right)}\vec{\eta}|f(\vec{\eta})|^{2}\dif r\dif\vec{\eta}\,.\]

\begin{lem}
In dimension $d\geq3$, \[
\big\|\mathcal{F}^{(h)}(\vec{\xi}^{\,(h)})-F^{(h)}(\vec{\xi}^{\,(h)})\big\|_{\infty}\leq C_{g}^{2}h^{\nu(d)-\delta}\]
with $0<\delta<\nu(d)=\frac{d-2}{4}$ and $C_{g}^{2}=4\|g\|_{L^{1}}+2\pi^{d/2}(C_{g}^{1}\|\hat{g}'\|_{L^{1}}+\|\hat{g}\|_{L^{1}})$.\end{lem}
\begin{proof}
Let $t$ be in $\mathbb{R}^{+}$ and $\Lambda>0$, by a change of
variable \begin{align*}
\big| & \mathcal{F}^{(h)}(\vec{\xi}^{\,(h)})(t)-F^{(h)}(\vec{\xi}^{\,(h)})(t)\big|\\
 & \leq4\int_{0}^{\Lambda}\int_{\mathbb{R}^{d}}|g(\vec{\eta})|\dif\vec{\eta}\dif r\\
 & \phantom{=}+2\int_{\Lambda}^{t/h}\Big|\int_{\mathbb{R}^{d}}e^{-ir(\vec{\eta}^{.2}-2\vec{\eta}.\vec{\xi}_{t}^{\,(h)})}\big(e^{2ir\vec{\eta}.(\fint_{t-hr}^{t}\vec{\xi}_{\sigma}^{\,(h)}\dif\sigma-\vec{\xi}_{t}^{\,(h)})}-1\big)g(\vec{\eta})\dif\vec{\eta}\Big|\dif r\\
 & \leq4\Lambda\left\Vert g\right\Vert _{L^{1}}\\
 & \phantom{=}+\!2\negmedspace\int_{\Lambda}^{t/h}\negmedspace\Big|\negmedspace\int_{\mathbb{R}^{d}}e^{-ir\vec{\eta}^{\prime.2}}\!\big(e^{2ir(\vec{\eta}'+\vec{\xi}_{t}^{\,(h)}).(\fint_{t-hr}^{t}\vec{\xi}_{\sigma}^{\,(h)}\dif\sigma-\vec{\xi}_{t}^{\,(h)})}-1\big)g(\vec{\eta}^{\,\prime}+\vec{\xi}_{t}^{\,(h)})\!\dif\vec{\eta}^{\,\prime}\Big|\!\dif r\end{align*}
as $\big|e^{ir|\vec{\xi}_{t}^{\,(h)}|^{2}}\big|=1$. We then control
the internal integral using Parseval's equality. Indeed we have the
Fourier transforms \[
\mathcal{F}(e^{-ir\vec{\eta}^{\prime.2}})(y)=\Big(\frac{\pi}{-ir}\Big)^{d/2}e^{i\frac{\vec{y}^{2}}{4r}}\,,\]
\[
\mathcal{F}\big(g(\vec{\eta}'+\vec{\xi}_{t}^{\,(h)})\big)(y)=e^{i\vec{\xi}_{t}^{\,(h)}.\vec{y}}\hat{g}(\vec{y})\]
and\begin{multline*}
\mathcal{F}\big(e^{2ir(\vec{\eta}'+\vec{\xi}_{t}^{\,(h)}).(\fint_{t-hr}^{t}\vec{\xi}_{\sigma}^{\,(h)}d\sigma-\vec{\xi}_{t}^{\,(h)})}g(\vec{\eta}'+\vec{\xi}_{t}^{\,(h)})\big)(\vec{y})\\
=e^{i\vec{\xi}_{t}^{\,(h)}.\vec{y}}\hat{g}\big(\vec{y}-2r\big(\fint_{t-hr}^{t}\vec{\xi}_{\sigma}^{\,(h)}d\sigma-\vec{\xi}_{t}^{\,(h)}\big)\big)\,.\end{multline*}
The difference between the two last Fourier transforms can be estimated
using Remark~\ref{rem:Moy_Xi-Xi}, and thus \[
\big\|\hat{g}(\vec{y}-2r(\fint_{t-hr}^{t}\vec{\xi}_{\sigma}^{\,(h)}d\sigma-\vec{\xi}_{t}^{\,(h)}))-\hat{g}(\vec{y})\big\|_{L^{1}}\leq2\min\left\{ C_{g}^{1}\|\hat{g}'\|_{L^{1}}r^{2}h,\|\hat{g}\|_{L^{1}}\right\} \,.\]
The internal integral is then controled as\begin{multline*}
\big|\int_{\mathbb{R}^{d}}e^{-ir\vec{\eta}^{\prime.2}}\big(e^{2ir(\vec{\eta}'+\vec{\xi}_{t}^{\,(h)}).(\fint_{t-hr}^{t}\vec{\xi}_{\sigma}^{\,(h)}\dif\sigma-\vec{\xi}_{t}^{\,(h)})}-1\big)g(\vec{\eta}'+\vec{\xi}_{t}^{\,(h)})\dif\vec{\eta}'\big|\\
\leq2\pi^{d/2}\big(C_{g}^{1}\|\hat{g}'\|_{L^{1}}+\|\hat{g}\|_{L^{1}}\big)\min(r^{2}h,1)\, r^{-d/2}\,.\end{multline*}
By interpolation, $\min(r^{2}h,1)\, r^{-d/2}\leq h^{\theta}r^{2\theta-\frac{d}{2}}$,
and we then obtain\[
\big\|\mathcal{F}^{(h)}(\vec{\xi}^{\,(h)})-F^{(h)}(\vec{\xi}^{\,(h)})\big\|_{\infty}\leq4\Lambda\|g\|_{L^{1}}+2\pi^{d/2}\big(C_{g}^{1}\|\hat{g}'\|_{L^{1}}+\|\hat{g}\|_{L^{1}}\big)h^{\theta}\Lambda^{2\theta-\frac{d}{2}+1}\,.\]
An optimization of the parameter $\theta$ and of $\mu$ in $\Lambda=h^{\mu}$
shows that we can take $\mu$ as high as $\frac{d-2}{4}-\delta$ for
any $\delta>0$. (We then get $\theta=\frac{\mu d}{2\left(1+2\mu\right)}$.)
This concludes the proof.\end{proof}
\begin{lem}
In dimension $d\geq3$, for $h,\, t>0$,\[
\big\| F^{(h)}(\vec{\xi}^{\,(h)})-F^{(0)}(\vec{\xi}^{\,(h)})\big\|_{\infty,[T,+\infty)}\leq C_{g}^{3}\Big(\frac{h}{T}\Big)^{\frac{d}{2}-1}\,,\]
with $C_{g}^{3}=\Big(\frac{d}{2}-1\Big)\pi^{d/2}\|\hat{g}\|_{L^{1}}$.\end{lem}
\begin{proof}
Indeed\[
\big|F^{(h)}(\vec{\xi}^{\,(h)})(t)-F^{(0)}(\vec{\xi}^{\,(h)})(t)\big|\leq\pi^{d/2}\|\hat{g}\|_{L^{1}}\int_{t/h}^{+\infty}\frac{\dif r}{r^{d/2}}\]
which gives the result.
\end{proof}
As a corollary we get: 
\begin{lem}
In dimension $d\geq3$, for $h,\, T>0$ and $0<\delta<\frac{d-2}{4}$
\[
\big\| F^{(0)}(\vec{\xi}^{\,(h)})-\vec{\xi}^{\,(h)\prime}\big\|_{\infty,[\sqrt{h},+\infty)}\leq C_{g}^{2}h^{\frac{d-2}{4}-\delta}+C_{g}^{3}h^{\frac{d-2}{4}}\leq\delta(h)\,,\]
\begin{align*}
\Gamma_{t}^{(h)} & =\int_{0}^{t}\big|F^{(0)}(\vec{\xi}_{s}^{\,(h)})-\vec{\xi}_{s}^{\,(h)\prime}\big|\dif s\\
 & \leq\sqrt{h}+t\delta(h)=\delta_{1}(h)\,.\end{align*}
with $\delta(h)\to0$ as $h\to0$. 
\end{lem}
We now compare $\vec{\xi}^{\,(0)}$ and $\vec{\xi}^{\,(h)}$.
\begin{prop}
In dimension $d\geq3$, let $T\in(0,T_{max})$ then $(\vec{\xi}_{t}^{\,(h)})_{t\in[0,T]}$
converges uniformly to $(\vec{\xi}_{t}^{\,(0)})_{t\in[0,T]}$ as $h\to0$.\end{prop}
\begin{proof}
Let $r,\, R>0$ and $\varepsilon>0$ be such that $r<|\vec{\xi}_{t}^{\,(0)}|-\varepsilon$
and $|\vec{\xi}_{t}^{\,(0)}|+\varepsilon<R$ for all $t\in[0,T]$.
Thus the open tubular neighbourhood \[
A_{\varepsilon}=\bigcup_{t\in[0,T]}B(\vec{\xi}_{t},\varepsilon)\]
of the trajectory $(\vec{\xi}_{t}^{\,(0)})_{t\in[0,T]}$ is included
in the compact ring \[
C(r,R)=\{\vec{x}\in\mathbb{R}^{d},r\leq|\vec{x}|\leq R\}\,.\]
On this ring the application $F^{(0)}$ is Lipschitz with the Lipschitz
constant $L$. We show that for $h$ small enough $\vec{\xi}_{t}^{\,(h)}$
remains in $A_{\varepsilon}$. We first observe that by continuity,
for small $t$, $\vec{\xi}_{t}^{\,(h)}$ remains in $A_{\varepsilon}$.
And that if, on an intervall $[0,t_{0})$, $\vec{\xi}_{t}^{\,(h)}$
remains in $C(r,R)$ then we can compute a Gronwall type estimate.
Let $G(t)=\int_{0}^{t}\big|\vec{\xi}_{s}^{\,(0)}-\vec{\xi}_{s}^{\,(h)}\big|\dif s$,
then \begin{align*}
G'\left(t\right) & =\big|\vec{\xi}_{t}^{\,(0)}-\vec{\xi}_{t}^{\,(h)}\big|\\
 & =\big|\int_{0}^{t}\big(F^{(0)}(\vec{\xi}_{s}^{\,(0)})-F^{(0)}(\vec{\xi}_{s}^{\,(h)})+F^{(0)}(\vec{\xi}_{s}^{\,(h)})-\vec{\xi}_{s}^{\,(h)\prime}\big)\dif s\big|\\
 & \leq LG(t)+\Gamma_{t}^{(h)}\,.\end{align*}
As $G(0)=0$ we get \[
G(t)\leq\int_{0}^{t}e^{L(t-s)}\Gamma_{s}^{(h)}\dif s\leq\int_{0}^{t}e^{Lt}\delta_{1}(h)\dif s\leq te^{Lt}\delta_{1}(h)\]
and \[
\big|\vec{\xi}_{t}^{\,(0)}-\vec{\xi}_{t}^{\,(h)}\big|=G'(t)\leq LG(t)+\Gamma_{t}^{(h)}\leq(Lte^{Lt}+1)\delta_{1}(h)\,.\]
for $t\in[0,t_{0})$. Let $h_{0}$ such that, for $h$ smaller than
$h_{0}$, $(LTe^{LT}+1)\delta_{1}(h)<\frac{\varepsilon}{2}$. Suppose
there exists a time $t$ in $[0,T]$ such that $\vec{\xi}_{t}^{\,(h)}\notin A_{\varepsilon}$
we define $t_{0}$ the inferior bound of such times. As $A_{\varepsilon}$
is open and $\vec{\xi}_{t}^{\,(0)}$ and $\vec{\xi}_{t}^{\,(h)}$
are continous then $\vec{\xi}_{t_{0}}^{\,(h)}\notin A_{\varepsilon}$.
But as $|\vec{\xi}_{t}^{\,(0)}-\vec{\xi}_{t}^{\,(h)}|\leq(LTe^{LT}+1)\delta_{1}(h)\leq\frac{\varepsilon}{2}$
on $[0,t_{0})$, by continuity $|\vec{\xi}_{t_{0}}^{\,(0)}-\vec{\xi}_{t_{0}}^{\,(h)}|\leq\frac{\varepsilon}{2}$
and so $\vec{\xi}_{t_{0}}^{\,(h)}\in A_{\varepsilon}$ we thus get
a contradiction and, for any $t\in[0,T]$ and $h\leq h_{0}$, $\vec{\xi}_{t}^{\,(h)}$
remains in $A_{\varepsilon}$. And thus we get the uniform convergence
of $(\vec{\xi}_{t}^{\,(h)})_{t\in[0,T]}$ to $(\vec{\xi}_{t})_{t\in[0,T]}$
as $h\to0$.
\end{proof}

\section*{Acknowledgement}

The author would like to thank Francis Nier for proposing the problem
which led to the study of those equations.

\bibliographystyle{plain}
\bibliography{biblio_xi}

\lyxaddress{}

\end{document}